\documentclass[12pt]{amsart}
\usepackage{amssymb,epsfig,amsmath,latexsym,amsthm}
\usepackage{graphicx}
\theoremstyle{plain}
\newtheorem{theorem}{Theorem}[section]
\newtheorem{lemma}[theorem]{Lemma}

\theoremstyle{definition}
\newtheorem{definition}[theorem]{Definition}
\newtheorem{remark}[theorem]{Remark}

\newtheorem{remarks}[theorem]{Remarks}

\DeclareMathOperator{\area}{area}

\DeclareMathOperator{\inte}{int}

\DeclareMathOperator{\mecd}{mecd}

\newcommand{\BN}{\mathbb N}

\newcommand{\mG}{\mathcal{G}}
\newcommand{\mH}{\mathcal{H}}

\usepackage{epsfig,color}

\makeatother

\theoremstyle{definition}

\def\fnum{equation}

\numberwithin{equation}{section}

\newcommand{\diam}{{\text {diam}}}

\begin{document}

\title{On the classification of Heegaard Splittings}

\author{Tobias Holck Colding}\address{Department of Mathematics\\ MIT\\Cambridge, MA 02139-4307}

\author{David Gabai}\address{Department of Mathematics\\Princeton
University\\Princeton, NJ 08544}

\author{Daniel Ketover}


\thanks{September 19, 2015\\The first author
was partially supported by NSF Grant DM  
1404540 and NSF FRG grant DMS 
 0854774, the second by DMS-1006553 and NSF FRG grant DMS-0854969 and the third by an NSF Postdoctoral Fellowship.}
 
 \email{colding@math.mit.edu, gabai@math.princeton.edu and dketover@math.princeton.edu}

\begin{abstract} The long standing classification problem in the theory of Heegaard splittings of 3-manifolds is to exhibit for each closed 3-manifold a complete list, without duplication, of all its irreducible Heegaard surfaces, up to isotopy.  We solve this problem for non Haken hyperbolic 3-manifolds.
  \end{abstract}

\maketitle

\setcounter{section}{-1}

\section{Introduction}\label{S0}  The main result of this paper is 

\begin{theorem}  \label{main} Let $N$ be a closed  non Haken hyperbolic 3-manifold.  There exists a constructible set $S_0, S_1, \cdots, S_n$ such that if $S$ is an irreducible Heegaard splitting, then $S$ is isotopic to exactly one $S_i$.\end{theorem}

\begin{remarks}  Given $g\in \BN$ Tao Li \cite{Li3} shows how to construct a finite list of genus-$g$ Heegaard surfaces such that, up to isotopy, every genus-$g$ Heegaard surface appears in that list.  By \cite{CG}  there exists a computable $C(N)$ such that one need only consider $g\le C(N)$, hence there exists a constructible set of Heegaard surfaces that contains every irreducible Heegaard surface.  However, this list may contain reducible splittings and duplications.  The main goal of this paper is to give an algorithm that weeds out the duplications and reducible splittings.     \end{remarks}

\noindent\emph{Idea of Proof.}  We first prove the \emph{Thick Isotopy Lemma} which implies that if $S_i$ is isotopic to $S_j$, then there exists a smooth isotopy by surfaces of area uniformly bounded above and diametric soul uniformly bounded below.  (The \emph{diametric soul} of a surface $T\subset N$ is the infimal diameter in $N$ of the essential closed curves in $T$.)  The proof of this lemma uses a 2-parameter sweepout argument that may be of independent interest.  We construct a graph $\mG$ whose vertices comprise a finite net in the set of genus $\le C(N)$ embedded surfaces of uniformly bounded area and diametric soul, i.e. up to small perturbations and pinching  spheres any such surface is close to a vertex of $\mG$. The edges of $\mG$ connect vertices that are perturbations of each other up to pinching necks.  Thus $S_i$ and $S_j$ are isotopic if and only if they lie in the same component of $\mG$.  For technical reasons the construction of $\mG$ is carried out in the PL category. 

The Thick Isotopy Lemma also shows that any reducible Heegaard surface $S_i$ is isotopic to an \emph{obviously} reducible one through surfaces of uniformly bounded area and diametric souls.  Thus $S_i$ is reducible if and only if it lies in the same component of $\mG$ as an \emph{obviously} reducible one. A surface is obviously reducible if it has small diameter essential curves that compress to each side.  The point here is that the surfaces in the isotopy look incompressible at a  small scale, but at a somewhat larger scale the isotopy ends at an obviously reducible surface.

Our main result for manifolds with taut ideal triangulations was earlier obtained by Jesse Johnson \cite{Jo}.  His methods inspired some of  those used in this paper.

This paper is organized as follows.  Basic definitions and facts are given in  \S1,  the Thick Isotopy Lemma is proved in \S2 and the graph $\mG$ is constructed in \S3.

\vskip 8pt
\begin{remark}  
We believe that the methods of this paper will have applications to the  classification problem for compact 3-manifolds and may also have application to the question of finding the minimal common stabilization to two splittings.  \end{remark}




\section{Heegaard splittings and Paths of Heegaard foliations}


\begin{definition}  A \emph{Heegaard splitting} of a closed orientable 3-manifold $M$ consists of an ordered pair  $(H_0, H_1)$ of handlebodies whose union is $M$ and whose intersection is their boundaries.  This common boundary $S$ is called a \emph{Heegaard surface}.  Two Heegaard splittings $(H_0, H_1)$, $(H'_0, H'_1)$ are \emph{isotopic} if there exists an ambient isotopy of $M$ taking $H_0$ to $H'_0$.\end{definition}

\begin{remark}  The Heegaard splitting  $(H_0, H_1)$ may not be isotopic to $(H_1, H_0)$, e.g. see \cite{Bir}. The ordering induces  a transverse orientation on the Heegaard surface $S$, pointing from $H_0$ to $H_1$ and isotopy is required to preserve it.   All the results of this paper naturally carry over to the weaker setting where isotopy of Heegaard surfaces need not preserve the transverse orientation.\end{remark}

\begin{definition}  A \emph{Heegaard foliation} $\mH$ is a singular foliation of $M$ induced by a submersion to $[0,1]$ such that a Heegaard surface $S$ is a leaf, and if $H_0$ and $H_1$ are the handlebodies bounded by $S$ then each $H_i$ has a PL spine $E_i$ such that $\mH|H_i\setminus E_i$ is a fibration by surfaces that limit to $E_i$.\end{definition}

\begin{lemma}\label{pathoffoliations}  i)  Any two Heegaard foliations $\mH, \mH'$ of the same Heegaard splitting $(H_0, H_1)$ which have the same spines are isotopic, via isotopies fixing $S$ and the spines pointwise.

ii)  Given $\mH_0, \mH_1$ two Heegaard foliations of the same Heegaard splitting $(H_0, H_1)$ there exists a path $\mH_t$ of Heegaard foliations from $\mH_0$ to $\mH_1$ which varies smoothly away from the spines and fixes the Heegaard surface $S$ throughout. 

iii)  Given isotopic Heegaard splittings $(H_0, H_1)$ and $(H'_0, H'_1)$ with Heegaard foliations $\mH_0, \mH_1$ and Heegaard surfaces $S_0, S_1$, there exists a path $\mH_t $ of Heegaard foliations from $\mH_0$ to $\mH_1 $ which varies smoothly away from the spines and takes the Heegaard surface $S_0$ to $S_1$.\end{lemma}

\emph{Proof of i)}  It suffices to show that $\mH|H_0$ is isotopic to $\mH'|H_0$ via an isotopy fixing $S$.  Let $S = S_1, S_2, \cdots,$ $\ S= S'_1, S'_2, \cdots$ be nested sequences of leaves of $\mH|H_0$ and  $\mH'|H_1$ which converge to $E_0$.  Next isotope $H_0$, fixing $S_1$ so that $S_i$ is taken to $S'_i$ and if $f:\cup S_i\cup E_0\to \cup S'_i\cup E_0$, then $f$ is continuous and $f|S_i$ is smooth. By adjusting the isotopy using the normal flow to the leaves we can assume that $\mH$ and $\mH'$ coincide near each $S_i$.  We abuse notation by denoting $f_*(\mH)$ by $\mH$.  Let $V_k$ denote the compact region bounded by $S_k$ and $S_{k+1}$.  By \cite{LB}  each $V_k$ can be ambiently isotoped by an isotopy that is fixed near $\partial V_k$, so that $\mH|V_k$ is taken to $\mH'|V_k$.  With care the union of these isotopies give a globally defined isotopy, i.e. it is continuous at $E_0$.\qed
\vskip 8pt 

\emph{Proof of ii)}  It suffices to construct a path from $\mH_0|H_0$ to $\mH_1|H_0$ which fixes $S$ and is smooth away from the spine.  Let $E_0$ (resp. $E_1$) denote the spine of $H_0$ associated to $\mH_0$ (resp. $\mH_1)$.  Since handlebodies have unique spines up to sliding of 1-cells, there exists a path $\mH_t, t\in [1/2,1]$ so the spine of $H_0$ associated to $\mH_{1/2}$ is equal to $ E_0$ and the nonsingular leaves vary smoothly.  The result now follows from i).\qed

\vskip 8pt
\emph{Proof of iii)}  An initial ambient isotopy parametrized by $ [2/3,1]$ takes $S_0$ to $S_1$, so by the covering isotopy theorem there exists a smoothly varying path of Heegaard foliations $\mH_t, t\in [2/3, 1]$ such that $\mH_{2/3}$ has $S_0$ as a Heegaard surface.  The result now follows from ii).\qed
\vskip 8pt

\section{Thick Isotopy Lemma}

The proof of the Thick Isotopy Lemma relies on a 2-parameter min-max argument to find a path of surfaces joining two isotopic Heegaard splittings with all areas bounded from above.  

We first introduce the min-max theory with $n$ parameters.  The min-max theory is due to Almgren-Pitts \cite{P} but we will use a refinement by Simon-Smith \cite{SS} that allows one to consider sweepouts of a fixed topology.   We also need the optimal genus bounds established in \cite{K}.  

Throughout this section, $M$ denotes a closed orientable 3-manifold, and $\mathcal{H}^2(\Sigma)$ denotes the 2-dimensional Hausdorff measure of a set $\Sigma\subset M$.  Set $I^n=[0,1]^n\subset\mathbb{R}^n$.  Let $\{\Sigma_t\}_{t\in I^n}$ be a family of closed subsets of $M$ and $B\subset\partial I^n$.   We call the family $\{\Sigma_t\}$ a \emph{(genus-g) sweepout} if

\begin{enumerate}
\item $\mathcal{H}^2(\Sigma_t)$ is a continuous function of $t\in I^n$,
\item $\Sigma_t$ converges to $\Sigma_{t_0}$ in the Hausdorff topology as $t\rightarrow t_0$.
\item For $t_0\in I^n\setminus B$, $\Sigma_{t_0}$ is a smooth closed surface of genus $g$ and $\Sigma_t$ varies smoothly for $t$ near $t_0$.
\item For $t\in B$, $\Sigma_t$ consists of a $1$-complex.
\end{enumerate}

Given a family of subsets $\{\Sigma_t\}_{t\in\partial I^n}$, we say that $\{\Sigma_t\}_{t\in\partial I^n}$ \emph{extends to a sweepout} if there exists a sweepout $\{\Sigma_t\}_{t\in I^n}$ that restricts to  $\{\Sigma_t\}_{t\in\partial I^n}$ at the boundary.  A \emph{Heegaard foliation} is a sweepout $\{\Sigma_t\}$ parameterized by $[0,1]$ where $\Sigma_t$ is a Heegaard surface for $t\neq 1,0$ and the sets $\Sigma_0$ and $\Sigma_1$ are $1$-complexes in the handlebodies determined by the Heegaard splitting.  Additionally, we assume that $\{\Sigma_t\}$ is a singular foliation (with only two singular leaves). If $\mathcal{H}_0(s)_{s=0}^{1}$ and $\mathcal{H}_1(s)_{s=0}^{1}$ are two Heegaard foliations with respect to isotopic Heegaard splittings, then we call $\{\Sigma_t\}_{t\in I^2}$ a \emph{Heegaard sweepout joining $\mathcal{H}_0$ to $\mathcal{H}_1$} if it is a sweepout such that for $s\in [0,1]$, $\Sigma_{(0,s)}=\mathcal{H}_0(s)$ and $\Sigma_{(1,s)}=\mathcal{H}_1(s)$ and also for $i\in\{0,1\}$, $\{\Sigma_{(t,i)}\}_{t=0}^1$ is a continuously varying family of 1-complexes. \\
\indent 
Beginning with a genus-g sweepout $\{\Sigma_t\}$ we need to construct comparison sweepouts which agree with $\{\Sigma_t\}$ on $\partial I^n$.  We call a collection of sweepouts $\Pi$ \emph{saturated} if it satisfies the following condition:  for any map $\psi\in C^\infty(I^n\times M,M)$ such that for all $t\in I^n$, $\psi(t,.)\in \mbox{Diff}_0(M)$ and $\psi(t,.)=id$ if $t\in\partial I^n$, and a sweepout $\{\Lambda_t\}_{t\in I^n}\in\Pi$ we have $\{\psi(t,\Lambda_t)\}_{t\in I^n}\in\Pi$.   Given a sweepout $\{\Sigma_t\}$, denote by $\Pi=\Pi_{\{\Sigma_t\}}$ the smallest saturated collection of sweepouts containing $\{\Sigma_t\}$.
We define the \emph{width} of $\Pi$ to be

\begin{equation}
W(\Pi,M)=\inf_{\{\Lambda_t\}\in\Pi}\sup_{t\in I^n} \mathcal{H}^2(\Lambda_t).
\end{equation}

A \emph{minimizing sequence} is a sequence of sweepouts $\{\Sigma_t\}^i\in\Pi$ such that 

\begin{equation}
\lim_{i\rightarrow\infty}\sup_{t\in I^n}\mathcal{H}^2(\Sigma_t^i)=W(\Pi,M).
\end{equation}

Finally, a \emph{min-max sequence} is a sequence of slices $\Sigma_{t_i}^i$, $t_i\in I^n$ taken from a minimizing sequence so that $\mathcal{H}^2(\Sigma_{t_i}^i)\rightarrow W(\Pi,M)$.
The main point of the Min-Max Theory of Almgren-Pitts \cite{P} is that if the width is bigger than the maximum of the areas of the boundary surfaces, then some min-max sequence converges to a minimal surface in $M$:

\begin{theorem} (Multi-parameter Min-Max Theorem)\label{minmax}
Given a sweepout $\{\Sigma_t\}_{t\in I^n}$ of genus $g$ surfaces, if
\begin{equation}
W(\Pi,M)> \sup_{t\in\partial I^n} \mathcal{H}^2(\Sigma_t)
\end{equation}
then there exists a min-max sequence $\Sigma_i:=\Sigma_{t_i}^i$ such that

\begin{equation}
\Sigma_i\rightarrow\sum_{i=1}^k n_i\Gamma_i \mbox{     as varifolds,} 
\end{equation}
where $\Gamma_i$ are smooth closed embedded minimal surfaces and $n_i$ are positive integers. Moreover, after performing finitely many compressions on $\Sigma_i$ and discarding some components, each connected component of $\Sigma_i$ is isotopic to one of the $\Gamma_i$ or to a double cover of one of the $\Gamma_i$.  We have the following genus bounds with multiplicity:

\begin{equation}\label{genusbound}
\sum_{i\in\mathcal{O}}n_ig(\Gamma_i)+\frac{1}{2}\sum_{i\in\mathcal{N}}n_i(g(\Gamma_i)-1)\leq g,
\end{equation}
where $\mathcal{O}$ denotes the subcollection of $\Gamma_i$ that are orientable and $\mathcal{N}$ denotes those $\Gamma_i$ that are non-orientable, and where $g(\Gamma_i)$ denotes the genus of $\Gamma_i$ if it orientable, and the number of crosscaps that one attaches to a sphere to obtain a homeomorphic surface if $\Gamma_i$ is non-orientable.
\end{theorem}\noindent
Theorem \ref{minmax} is proved in the Appendix.  We can now state the main application of the min-max theory in our setting:

\begin{theorem} \label{sweepout} Let $N$ be a closed hyperbolic 3-manifold and let $\{\Sigma_t\}_{t\in I^n}$ be a genus-g sweepout.  Set $C=\max(\sup_{t\in\partial I^n}\mathcal{H}^2(\Sigma_t), 2\pi(2g-2))$.  Then for all $\epsilon>0$, there is a sweepout $\{\Lambda_t\}_{t\in I^n}$ extending $\{\Sigma_t\}_{t\in\partial I^n}$ such that $\sup_{t\in I^n}\mathcal{H}^2(\Lambda_t)\le C+\epsilon$.\end{theorem}

The following is an immediate corollary of Theorem \ref{sweepout}.

\begin{lemma} \label{controlled area}  Let $N$ be a closed hyperbolic 3-manifold and let $\mH_0$ and $\mH_1$ be isotopic Heegaard foliations on $N$ representing Heegaard surfaces of genus $g$.  Denote the leaves of $\mH_i$ by $\mH_i(t), t\in [0,1]$.  Let $C$ denote $\max\{2\pi(2g-2), \max\{\area(\mH_s(t))|s\in \{0,1\}, t\in [0,1]\}\}$.  Then for all $\epsilon>0$, $\mH_0, \mH_1$ extend to a sweepout $\{\Sigma_t\}_{t\in I^2}$ such that $\sup_{t\in I^2}\mH^2(\Sigma_t)\le C+\epsilon$.\end{lemma}

\noindent\emph{Proof of Theorem \ref{sweepout}.} 
By Lemma \ref{pathoffoliations}, we can find a sweepout $\{\Sigma_t\}_{t\in I^n}$ extending $\{\Sigma_t\}_{t\in\partial I^n}$.  Denote by $\Pi$ the saturation of sweepouts containing $\{\Sigma_t\}_{t\in I^n}$.  We argue by contradiction. Thus assume there exists an $\epsilon>0$ so that all sweepouts in $\Pi$ have a slice with area greater than $\max(2\pi(g-2),\sup_{t\in\partial I^n}\mathcal{H}^2(\Sigma_t))+\epsilon$.  Then the width $W(\Pi,N)$ of $\Pi$ satisfies:  

\begin{equation}\label{width}
W(\Pi,N)\geq\max(2\pi(2g-2),\sup_{t\in\partial I^n}\mathcal{H}^2(\Sigma_t))+\epsilon.  
\end{equation}

In particular \eqref{width} implies that $W(\Pi,N)$ is strictly greater than the area of all boundary surfaces.  Thus the Min-Max Theorem \ref{minmax} applies to give a positive integer combination of minimal surfaces  $\Gamma=\sum_i n_i\Gamma_i$, so that $W(\Pi,N)=\sum n_i\mH^2(\Gamma_i)$.   The area of a genus $g$ surface in a hyperbolic $3$-manifold is at most $2\pi (2g-2)$ and the area of a non-orientable surface is at most $2\pi (k-2)$ where $k$ is the number of cross-caps one must add to a sphere to obtain a homeomorphic surface.  Using these bounds, along with the genus bound \eqref{genusbound} and the fact that the multiplicity $n_i$ of a non-orientable surface occuring among the $\Gamma_i$ is even we obtain:
$$W(\Pi,N) =\sum n_i \mathcal{H}^2(\Gamma_i)\leq 2\pi\sum_{i\in\mathcal{O}} n_i(2g_i-2)+2\pi\sum_{i\in\mathcal{N}} n_i(g_i-2)\leq 2\pi(2g-2).$$  This contradicts \eqref{width}.
\qed
\\
\\
\noindent

To prove the main result of this section Lemma \ref{thick isotopy} we need to bound area from above and diametric soul from below.  The min-max theory enabled the area bound.  The following topological arguments will enable the diametric soul bound.  

\begin{definition}  Let $M$ be a Riemannian 3-manifold,  $S$ a surface embedded in $M$.  We say that $S$ is $\delta$-compressible if some essential simple closed curve of diameter $<\delta$  bounds an embedded disc in a closed complementary region.  Otherwise we say that  $S$ is $\delta$-\emph{locally incompressible}.  

Let $S$ be a closed separating surface in the 3-manifold $M$.  We say that $S$ is $\delta$-\emph{bi-compressible} if there exist essential simple closed curves in $S$ of diameter  $<\delta$ that respectively bound discs in each closed complementary region.  \end{definition}

\begin{lemma}   \label{local relations} Suppose that $M\neq S^3$.  If the Heegaard surface $S\subset M$ is $\eta$-bi-compressible, where $\eta<\delta_0/4$ and $\delta_0$ is the injectivity radius of $M$, then $S$ is weakly reducible. \end{lemma}

\begin{proof}   Either there exist disjoint balls of diameter $<2\eta$ containing compressible curves of each handlebody or a single ball $E$ of diameter $<4\eta$ contains such compressions.  The result is immediate in the former case.  In the latter case some essential curve $\alpha$ of $S\cap \partial E$ compresses in one of the handelbodies and that together with one of the other compressions gives the weak reduction.\end{proof} 

\begin{definition} Let $\mH$ be a Heegaard foliation with leaves parametrized by $ [0,1]$.  If $t\in (0,1)$, then the \emph{$H_0$ (resp. $H_1$) side} of $\mH(t)$ is the component of the closed complement that contains $\mH(0)$, (resp. $\mH(1)$).  More generally if $T_t$ is a 2-dimensional sweepout which coincides on $\partial I\times I$ with a path of Heegaard foliations, then for $v\in [0,1]\times (0,1)$ use the smooth variance to define the $H_0$ and $H_1$ sides of $T_v$.  It therefore makes sense to say that a curve in $T_v$ \emph{compresses to the $H_0$ or $H_1$ side}. 

Define $\mecd_0(T_v)$ (resp. $\mecd_1(T_v) $) (the \emph{minimal essential compressing diameter}) to be the minimal diameter of a curve in $T_v, v\in I\times (0,1)$ that compresses to the $H_0$ (resp. $H_1$) side.  \end{definition}

Since the surfaces $T_v$ vary continuously in $v$ we obtain: 

\begin{lemma} \label{mecd continuity}  The functions $\mecd_0$, $\mecd_1$ are continuous and extend to continuous functions on $I\times I$.\qed\end{lemma}

\begin{lemma}  \label{avoidance} Let $\mH$ be a Heegaard foliation in the Riemannian manifold $M$ and let $t\in (0,1)$.     Suppose that $\delta<\delta_0/16$ where $\delta_0$ is the injectivity radius.  If $1>w>t$ and $\mecd_0(\mH(w))<2\delta$, then either $\mecd_0(\mH(t))<4\delta$ or $\mH(w)$ is $4\delta$-bi-compressible.  \end{lemma}

\begin{proof} Let $T_s$ denote $\mH(s)$ and hence $\mecd_0(T_w)< 2\delta$.  Let $B$ be a smooth ball of diameter $<4\delta$, transverse to both $T_w$ and $T_t$, such that some essential curve $\gamma\subset T_w\cap B$  compresses to the $H_0$ side.  Among the components of $T_w\cap \partial B$ that are essential in $T_w$, choose an innermost one.   This curve $\alpha$ has diameter $<4\delta$ and compresses to either the $H_0$ or $H_1$ sides.  If it compresses to the $H_1$ side we are done.  Otherwise let $D$ be a compression disc for $\alpha$ that lies the $H_0$ side of $T_w$.  It can be chosen so that $D\cap T_t\subset \partial B$.  Among components of $T_t\cap D$ that are essential in $T_t$ let $\beta $ be an innermost one.  Since $T_t$ and $T_w$ cobound a product such a $\beta$ exists.  This $\beta$  has diameter $<4\delta$.  \end{proof}

\begin{definition}  A \emph{C-isotopy} $F:T\times I\to M$ is an isotopy such that for all $t$, $\area(T_t)\le C$.\end{definition}



\begin{lemma} \label{thick isotopy} (\textbf{Thick Isotopy Lemma})  Let $N$ be a closed non Haken hyperbolic 3-manifold with injectivity radius $\delta_0$ and let $\delta<\delta_0/16$.

i)  If $T_0$ and $T_1$ are  isotopic strongly irreducible genus-$g$ Heegaard surfaces that are $8\delta$-locally incompressible, then there exists a computable $C>0$ and  there exists a
$C$-isotopy $F:T\times I\to N$ from $T_0$ to $T_1$ with each $T_t$ is $\delta$-locally incompressible.

ii)  If $T_0$ is weakly reducible and  $8\delta$-locally incompressible, then there exists a computable $C>0$ and a $C$-isotopy from $F:T\times I\to N$ from $T_0$ to a $T_1$ such that each $T_t$ is  $\delta$-locally incompressible and $T_1$ is $8\delta$-locally bi-compressible.\end{lemma}

\noindent\emph{Proof of i).}  Given a Heegaard surface $T\subset N$, Haken's theory of hierarchies and normal surface theory provides an algorithm for showing that 
each side of $T$ is a handlebody and hence gives an algorithm for constructing a Heegaard foliation with $T$ as a leaf.  Applying these algorithms to 
$T_0$ and $T_1$ produces foliations $\mH_0$ and $\mH_1$ where for $i=0,1, T_i$ is identified with $\mH_i(1/2)$.  Let $C'$ denote an upperbound for the area of any leaf of either $\mH_0$ or $\mH_1$.  
Let $C=\max\{(7(2g-2), C'\}+1$.  Apply Lemma \ref{controlled area} to obtain a sweepout between $\mH_0$ and $\mH_1$.  Denote a (possibly singular) surface in this sweepout by either $T_{(s,t)}$ or $T_v$ where $v\in I\times I$.

Let $G_p=\mecd_0^{-1}([0,p))$ and $R_q=\mecd_1^{-1}([0,q))$.  Now $G_{1.5\delta}$ is an open set containing $[0,1]\times 0$ and its closure is disjoint from $0\times [1/2,1]$ by Lemmas \ref{avoidance} and \ref{local relations}.  A similar statement holds for $R_{1.5\delta}$.  Since $T_0$ is strongly irreducible, it follows that $\bar G_{1.5\delta}\cap \bar R_{1.5\delta}=\emptyset$.  Hence there exists compact smooth submanifolds $G$ and $R$ of $I\times I$  transverse to $\{0,1\}\times I$ such that $\bar G_\delta\subset\inte( G)\subset G\subset G_{1.5\delta}$ and $\bar R_\delta\subset \inte(R)\subset R\subset R_{1.5\delta}$.  Again $R\cap G=\emptyset$ and in particular $G\cap I\times 1=\emptyset$.  Elementary topological considerations imply that   some component of $\partial G$ has endpoints in both $0\times I$ and $1\times I$.  A path in $I\times I$ starting at $(0,1/2)$ giving rise to an isotopy satisfying the conclusion of i) is obtained by concatenating arcs that lie in $\partial G$ and $0\times I$.  

\vskip10pt

\noindent\emph{Proof of ii).}  Since $T_0$ is weakly 
reducible it's reducible \cite{CaGo} and hence $T_0$ is isotopic to a stabilization $T_1$ of some strongly irreducible surface $R$.  Now $T_0$ is explicitly given and $R$ is isotopic to a surface in a known finite set of surfaces \cite{CG}, thus there exist Heegaard foliations $\mH_i$ extending $T_i$, $i=0,1$ such that a computable $C'$ bounds the area of any leaf of $\mH_0$ or $\mH_1$.  Since $T_1$ is a stabilization we can also assume that  for
$t\neq 0,1, \ \mH_1(t)$ is  $\delta$-bi-compressible.  Parametrize $\mH_0$ so that $T_0=\mH_0(1/2)$.

Construct the region $G$ as in i), though here $I\times 0\cup 1\times I\subset G$ but $(0,1/2)\notin G$.  Let $\sigma\subset I\times I$ be the maximal embedded path  starting at $(0, 1/2)$ whose interior is disjoint from $0\times [1/2,1]$ such that 
$\sigma\subset 0\times I \cup \partial G$ and $ \sigma \cap \inte(G)=\emptyset$.  By construction $\mecd_0(T_v)>\delta$ for $v\in \sigma$ and the terminal endpoint of $\sigma$ lies in $I\times 1\cup 0\times I$.  Let $\lambda$ be the maximal subpath of $\sigma$ 
starting at $(0,1/2)$ such that for $v\in \lambda,\ \mecd_1(T_v)\ge \delta$.  If $\lambda$ is a proper subpath of 
$\sigma$, then the desired isotopy is parametrized by $\lambda$.  Indeed, being proper the terminal endpoint $\omega$ is disjoint from $0\times [1/2,1]$.  If $\omega\in 0\times [0,1/2)$, then Lemma \ref{avoidance} implies that $T_\omega$ is $4\delta$-bicompressible.  If $\omega\cap 0\times I=\emptyset$, then $T_\omega$ is $1.5\delta$-bicompressible.  If $\lambda=\sigma$, then its terminal endpoint $\omega\in 0\times (1/2, 1)\cap \partial G$ and hence $\mecd_0(\omega)\le 1.5\delta$.  Lemma \ref{avoidance} now implies that $T_\omega$ is $4\delta$-bicompressible.\qed

\vskip10pt


\section{Weeding out duplications and reducibles}

Given a triangulated non Haken 3-manifold N and $g\in \BN$, Tao Li \cite{Li3}  gives an algorithm to construct a  set $\{S_0, S_1, \cdots, S_n\}$ of genus-$g$ Heegaard surfaces such that any genus-$g$ Heegaard surface is isotopic to one in this set.  In this section,  assuming that the manifold $N$ is hyperbolic, we give an algorithm to eliminate all the reducible splittings and duplications. Here is the idea.  Suppose that $S_0$ is weakly reducible.  The \emph{thick isotopy Lemma} shows that there is an isotopy from $S_0$ to one that is obviously weakly reducible, i.e there exist reducing curves that lie in small 3-balls, via an isotopy is through surfaces of area uniformly bounded above and injectivity radius uniformly bounded below.  (The lower bound being  smaller than the diameter of those small 3-balls.)  Ignoring long fingers and spheres that can be pinched off, one can find a finite net for the totality of such surfaces.  Construct a graph $\mG$ whose verticies are the points of the net and whose edges correspond to surfaces that differ by small perturbations.  Thus if $S_0$ is reducible, then it is in the same component of $\mG$ as an obviously reducible one.  This shows how to eliminate reducible splittings from $\{S_0, \cdots, S_n\}$.  A similar argument shows that if $S_i$ and $S_j$ are irreducible and isotopic then they are in the same component of $\mG$.  We carry out this idea  in the PL setting.  Here smooth isotopies are approximated by PL ones,  the relation of normal isotopy gives a net among surfaces transverse to a triangulation (after certain spheres are pinched off), and pinching and elementary isotopies across   faces of a triangulation give the edges of the graph.

\begin{definition}\label{combinatorial} Let $\Delta$ be a triangulation of the smooth 3-manifold $M$.  By a surface $T$ \emph{transverse} to $\Delta$ we mean that $T$ is transverse to the various skeleta of $\Delta$.  An isotopy $F:T\times [n,m] \to M$ between embedded surfaces $T_n$ and $T_m$ is said to be a \emph{generic $\Delta$-isotopy} if for all but finitely many times $n<t_1<t_2<\cdots<t_n<m$ each $T_t$ is transverse to $\Delta$.  Furthermore, for $\epsilon$ sufficiently small the passage from $T_{t_i-\epsilon}$ to $T_{t_i+\epsilon}$ is one of the following four elementary moves or their inverses.
\vskip 8pt

0) Passing through a vertex: See Figure A

1) Passing through an edge:  See Figure B

2) Passing through a face (two possibilities):  See Figures Ca, Cb
\vskip 8pt

\begin{figure}[ht]
\includegraphics[scale=0.60]{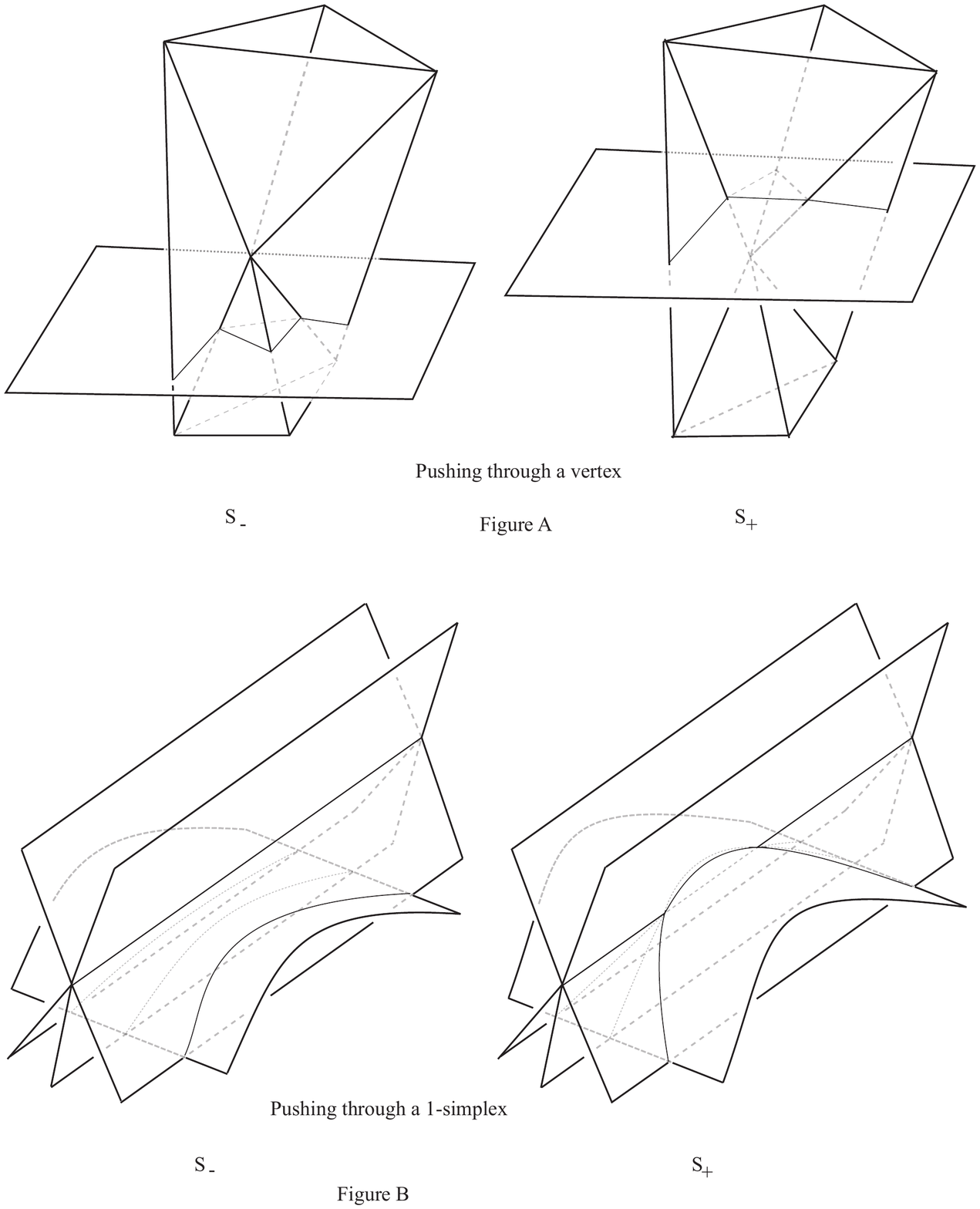}
\end{figure}
\begin{figure}[ht]
\includegraphics[scale=0.60]{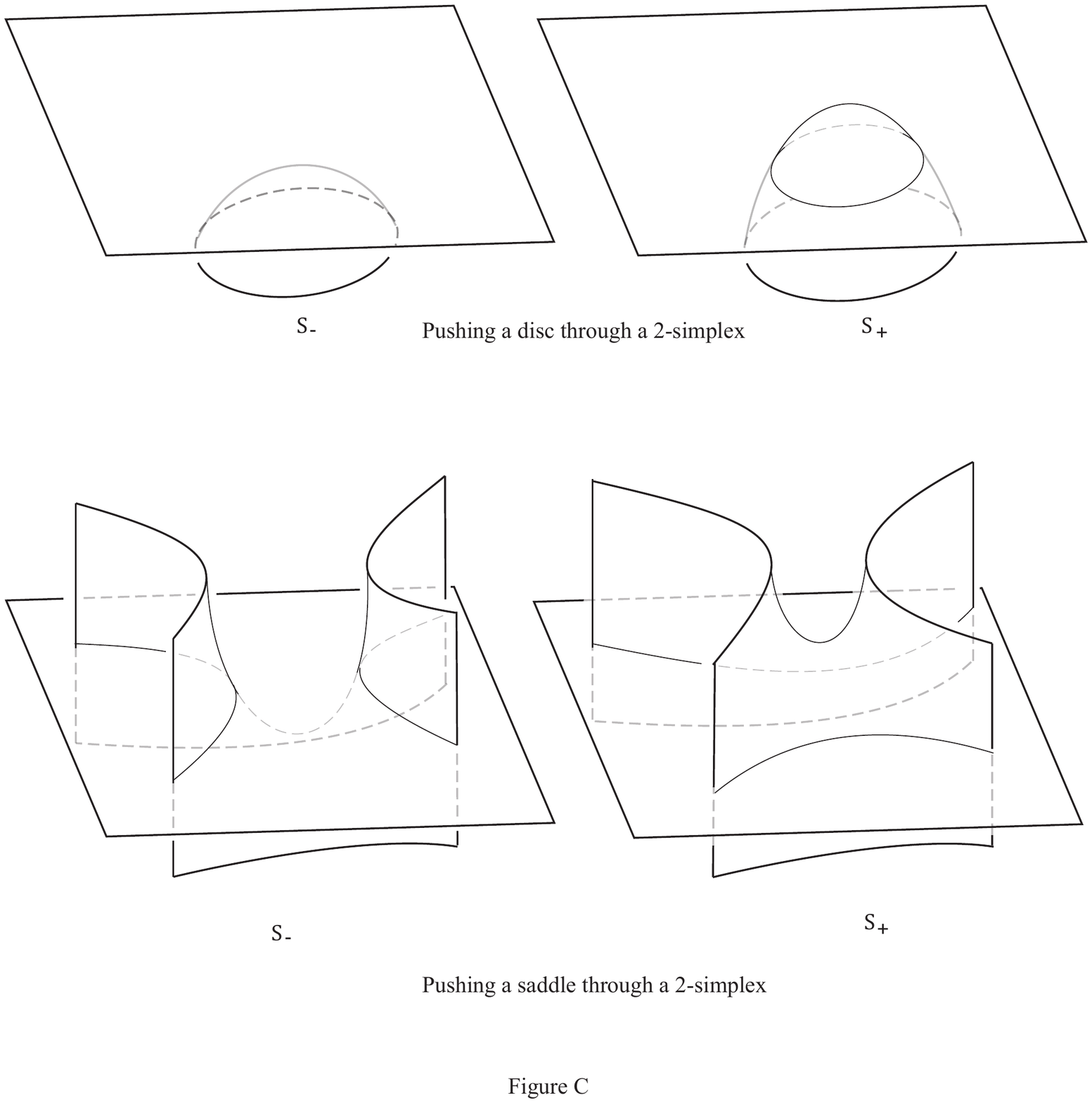}
\end{figure}

The next result follows from a standard perturbation argument.

\begin{lemma} \label{coarea} Let $\Delta$ be a triangulation of the Riemannian 3-manifold $N$ with metric $\rho$, $L>0$ and $\epsilon>0$, then there exists $K(\Delta, L, \epsilon, \rho)>1$ such that if $T$ is a closed embedded surface with $\area(T)<C$, then $T$ is isotopic to a surface $T'$ such that $|T'\cap \Delta^1|<KC$ and the diameter of the track of any point of the isotopy is at most $\epsilon$.  

If $F:T\times [0,1]\to M$ is a $C$-isotopy between surfaces  $T_0$ and $T_1$ that are transverse to $\Delta$ of weight at most $L$, then there exists a generic $K(C+1)$-$\Delta$-isotopy $G$ from $T_0$ to $T_1$ such that for all $x\in T$ and $t\in [0,1]$, $d(G(x,t), F(x,t))<\epsilon$.  \qed\end{lemma}

The transverse surface $T$ is said to be obtained by \emph{tubing} the transverse surfaces $P$ and $Q$ if  there exists a 3-simplex $\sigma$ and an embedded $D^2\times I \subset \inte(\sigma)$ such that $D^2\times 0\subset P$,  $D^2\times 1\subset Q$ and $T= (P\cup Q \cup \partial D^2\times I)\setminus (\inte(D^2)\times \{0,1\})$.  Two transverse surfaces $T$ and $T'$  \emph{differ by a pinch} if $T'$ is obtained by tubing $T$ and a 2-sphere or vice versa.  

A  \emph{pinched isotopy} $F$ from $T_0$ to $T_1$ consists of $0=n_1<\cdots<n_k=1$ and isotopies $f_i:T\times 
[n_i, n_{i+1}]$, $i=1, 2, \cdots, n_{k-1}$ such that $f_{n_1}(T,0)=T_0, f_{n_{k-1}}(T,1)=T_1$ and for all $i<k-1$ the 
surfaces $f_i(T, n_{i+1})$ and $f_{i+1}(T, n_{i+1})$ differ by a pinch.  \end{definition}





\begin{definition}  \label{crude} Let $\Delta$ be a triangulation of the 3-manifold $M$.  We say that $T$ is \emph{crudely almost normal} if $T$ is transverse to $\Delta$ and satisfies the following additional properties.  If $\tau$ is a 2-simplex, then no component of $T\cap \tau$ is a simple closed curve and if $\sigma$ is a 3-simplex, then each component of $T\cap\sigma$ is either a disc or an unknotted annulus.  Also there is at most one annulus component of $T\cap\sigma$. If all the components of intersection with 3-simplices are discs, then $T$ is said to be \emph{crudely normal}.   The $ \Delta$-transverse surfaces $T$ and $T'$ are said to be \emph{normally isotopic} if they are isotopic through a path of $\Delta$-transverse surfaces. Define the \emph{weight} of $T$ to be $|T\cap \Delta^1|$.   \end{definition}

\begin{lemma}  \label{crude finite} Let $\Delta$ be a triangulation of the 3-manifold $M$. Given $N\in \BN$, there are only finitely many normal isotopy classes of crudely normal and crudely almost normal surfaces of weight at most $N$.  \qed\end{lemma}

\begin{definition} A \emph{generic crudely normal} isotopy between two crudely almost normal surfaces $T_0$ and $T_1$ in a manifold with triangulation $\Delta$ is a generic $\Delta$-isotopy $F:T\times [0,1]\to N$ such that away from the finitely many non transverse points, the surfaces $F(T,t)$ are crudely almost normal.  A \emph{$C_1-\Delta$-isotopy} is one whose transverse surfaces all have weight $\le C_1$ with respect to $\Delta$.\end{definition}

The following lemma more or less says that provided the diameters of the simplices of $\Delta$ are sufficiently small, a generic pinched isotopy between crudely normal surfaces through $\delta$-locally incompressible surfaces  can be replaced by one whose generic interpolating surfaces are crudely normal without increasing the uniform upper bound on the weights of the interpolating surfaces.  

\begin{lemma}  \label{crude isotopy}  Let $M$ be a closed irreducible Riemannian 3-manifold with injectivity radius $>\delta_0$.  Let $\delta<\delta_0/16$ and $\Delta$ a triangulation on M such that if $\kappa\in\Delta$, then $st^2(\kappa)$ and $st(\kappa)$ are 3-balls of diameter $<\delta$.  Let $G$ be a generic pinched $C_1$-$\Delta$-isotopy between the crudely normal genus-$g$ surfaces $T_0$ to $T_1$ such that for each $t,\ G(T,t):=T_t$ is $\delta$-locally incompressible.  (At the pinch times there are two possibilities for $T_t$.)  Then there exists a generic pinched crudely normal $C_1$-$\Delta$-isotopy $H$ from $T_0$ to $T_1$. \end{lemma}

\noindent\emph{Proof.}  If $S_t$ is transverse to $\Delta$, then let $T_t$ be the crudely normal surface obtained as follows.  First, let $S'_t$ be the surface obtained by compressing $S_t$ near each component of $S_t\cap \partial \sigma$ for every 3-simplex $\sigma$ and then deleting the components disjoint from $\Delta^2$.   I.e. if $N(\Delta^2)$ is a small regular neighborhood of $\Delta^2$, then $S'_t\cap N(\Delta^2)=S_t\cap N(\Delta^2)$ and  for all every 3-simplex $\sigma$ each component of $S'_t\cap \sigma$ is a disc.  Since $S_t$ is $\delta$-locally incompressible, each component of $S_t\cap \partial \sigma$ is inessential in $S_t$.  It follows that $S'_t$ is a union of 2-spheres and a single component $T_t$ of genus-$g$.  Note that $T_t$ is crudely normal and up to normal isotopy $T_t$ is locally constant in $t$.  If $t$ is a pinch time, the result of  this construction is independent of the two possible choices for $S_t$. Finally  the local incompressibility condition implies that if $\sigma$ is a 3-simplex, then  any  closed curve in $S'_t\cap \sigma$   is inessential in $S'_t$.  

Away from a small neighborhood of the non transverse times, we let $H(S_t) = T_t$.  To complete the proof, it suffices to show that if $S_s$ is not transverse to $\Delta$ and $\epsilon$ is sufficiently small, then there is a generic pinched crudely normal $C_1-\Delta$-isotopy from $T_{s-\epsilon}$ to $T_{s+\epsilon}$.  We abuse notation by letting $S_{s-\epsilon}$ (resp. $S'_{s-\epsilon}$, $T_{s-\epsilon}$) be denoted $S_-$ (resp. $S'_-, T_-)$ with analogous notation  for $S_{s+\epsilon}$ (resp. $S'_{s+\epsilon}$, $T_{s+\epsilon})$.  
\vskip 8pt

\noindent\emph{Case 0:  $S_s$ passes through a vertex.}
\vskip 8pt

\noindent\emph{Proof of Case 0.}  Let $v$ denote the vertex. Let $\sigma$ be a 3-simplex having $v$ as a vertex.  We say an edge of $\sigma$ with vertex $v$ is \emph{up} if it lies to the $S_+$ side of $S$ and \emph{down} otherwise.  If $\sigma$ has three down (resp. up) edges, then $S_+\cap\sigma$ differs from $S_-\cap\sigma$ by the loss (resp. gain) of a normal triangle.  If it has one or two down edges, then exactly one component is non normally isotoped in the passage from $S_-$ to $S_+$.   It follows that $S'_+$ differs from $S'_-$ by an elementary move of type 0) and either $T_+$ is normally isotopic to $T_-$ or they differ by a move of type 0).\qed
\vskip 8pt

\noindent\emph{Case 1: $S_s$ is tangent to a 1-simplex.}  
\vskip 8pt

\noindent\emph{Proof of Case 1.}  Let $\eta$ denote the 1-simplex.  We can assume that the isotopy from $S_-$ to $S_+$ is as in Figure 4 as opposed to the inverse operation.  Let $\kappa_1, \cdots, \kappa_m$ denote the 2-simplices containing $\eta$.   We say that $\kappa_i$ is a \emph{down} simplex if some component of $\kappa_i\cap S_-$  splits into two components of $\kappa_i\cap S_+$, otherwise it is an \emph{up simplex}.   By perturbing the isotopy slightly, at a cost of creating moves of type 2b), we can assume that there exists a unique down simplex. It follows that $S'_+$ is obtained from $S'_-$ by a type 1) move and that either $T_+$ is normally isotopic to $T_-$ or they differ by a type 1) move.\qed
\vskip 8pt

\noindent\emph{Case 2a:  $S_s$ intersects a 2-simplex   locally at a point.}
\vskip 8pt

\noindent\emph{Proof of Case 2a.}  Let $\tau$ denote the 2-simplex.  Again we can assume that the isotopy is as in Figure 5 as opposed to the inverse one.  Up to normal isotopy $S'_+$ is equal to the union of $S'_-$ and a  2-sphere that intersects $\Delta^2$ in a single circle and $T_-=T_+$. \qed
\vskip 8pt

\noindent\emph{Case 2b:  $S_s$ intersects a 2-simplex at a saddle.}  
\vskip8pt

\noindent\emph{Proof of Case 2b.}  Let $\tau$ denote the 2-simplex, $x\in \tau$ the point of tangency and $\sigma_1$ and $\sigma_2$ the 3-simplices that contain $\tau$.  We can assume that $\sigma_1$ (resp. $\sigma_2$) lies \emph{below} (resp. \emph{above}) $x$, i.e. if a normal vector to $S_s$ based at $x$ points from $\sigma_1$ to $\sigma_2$, then the isotopy locally moves $S_s$ in that direction.  Let $\gamma$ be the component of $S_s\cap \tau$ that contains $x$.  We have various subcases.  
\vskip 8pt

\noindent\emph{Case 2bi:  $\gamma\cap \Delta^1=\emptyset$.}  
\vskip 8pt

\noindent\emph{Proof.}  By replacing the isotopy by the reverse if necessary, it suffices to consider the case that two $S^1$ components of $S_-\cap \tau$ transform to one component of $S_+\cap \tau$.  Up to normal isotopy $S'_-$  is the union of $S'_+$ and a $S^2$ component that intersects $\Delta^2$ in a single circle and $T_-=T_+$.  \qed
\vskip 8pt

\noindent\emph{Case 2bii: $|\gamma\cap \Delta^1|=2$.}  We can assume that an arc and an $S^1$ component of $S_-\cap \tau$ coalesce to an arc component of $S_+$ and therefore have the same conclusion as Case 2bi.\qed
\vskip 8pt

\noindent\emph{Case 2biii: $|\gamma\cap \Delta^1|=4$.}  Here two arc components $p_1, q_1$ of $S_-\cap \tau$ transform to two other arc components $p_2, q_2$ of $S_+\cap \tau$.  Observe that $p_1, q_1$ belong to the same component of $S_-\cap \partial \sigma_1$ if and only if $p_2, q_2$ belong to different components of $S_+\cap \partial \sigma_1$.  A similar fact holds for $\partial \sigma_2$.

We next show that it suffices to consider the case that $p_1,q_1$ are contained in  different components of both $S_-\cap \partial \sigma_1$ and $S_-\cap \partial \sigma_2$.  If they belong to the same components, then $p_2,q_2$ are contained in different components of both $S_+\cap \partial \sigma_1$ and $S_+\cap \partial \sigma_2$, so the reverse isotopy from $S_+$ to $S_- $ has the desired feature.  If $p_1,q_1$ belong to different components $p,q$ of $S_-\cap \partial \sigma_1$ but the same component of $S_-\cap \partial \sigma_2$, then $p\subset \sigma$ and is homologically essential, contradicting our local incompressibility condition.  

Now assume that $p_1,q_1$ are contained in  different components of both $S_-\cap \partial \sigma_1$ and $S_-\cap \partial \sigma_2$.  Let $P$ be the 
component  of $S'_-$ containing $p_1$ and $Q$ the component containing $q_1$.  Again if $P=Q$, then $p$ is homologically essential in $S_-$ contradicting $\delta$-local incompressibility.   Therefore one of them, say $Q$ is a 2-sphere.  If $P$ is also a 2-sphere, then $T_-$ is normally isotopic to $T_+$.  If $P=T_-$, 
then let $T'$ be the crudely almost normal surface obtained by tubing $P$ and $Q$, i.e. $T'$ is pinch equivalent to $T_-$.  Finally $T_+$ is obtained from 
$T'$ by a type 2b) move and the proof of the lemma is complete.\qed
\vskip 10pt


\noindent\emph{Proof of Theorem \ref{main}}   The initial data for $N$ is a totally geodesic triangulation $\Delta_1$.  (E.g. see \cite{Ma}.) Use it to find a lower bound $\delta_1$ for the injectivity radius of $N$.  By \cite {CG} there exists a computable $C(N)$ bounding the genus of any irreducible Heegaard surface.  Now fix $g\le C(N)$ and apply  Tao Li's algorithm \cite {Li3} (or \cite {CG}), to construct a set $\{S_1, \cdots, S_n\}$ of Heegaard surfaces  such that any Heegaard surface of genus-$g$ is isotopic to one of these surfaces. By construction the surfaces from \cite {Li3} are almost normal with respect to $\Delta_1$.  Next, pass to the second barycentric subdivision $\Delta_2$ of $\Delta_1$.  It follows that after  a small isotopy of  $\Delta_2$, each $S_i$ is normal to $\Delta_2$ and  $st(\kappa) $ and $st^2(\kappa)$ are closed 3-balls for each simplex $\kappa\in \Delta_2$.  Let $\delta_2\le \delta_1$ be such that any essential curve of any $S_i$ has diameter $>\delta_2$.   Let $\delta_3\le \delta_2$ be such that if $X\subset N $ and $\diam(X)\le \delta_3$, then $X\subset st(\kappa)$ some $\kappa\in \Delta_2$.  Let $\delta_4\le \delta_3/100$ and  subdivide $\Delta_2$ to $\Delta_3$ so that $\diam(\kappa)\le \delta_4$ for all $\kappa\in \Delta_3$ and so that each $S_i$ is normal.  Next let $L$ be the maximal $\Delta_3$ weight of all the $S_i$'s, $C=\max\{\area(S_1), \cdots, \area(S_n), 7(2g-2)\}$ and  $K=K(\Delta_3, \delta_4/2, L)$ as in Lemma \ref{coarea}.   Note that $K\ge L$.  


Construct a graph $\mG$ whose vertices are normal isotopy classes of $ \Delta_3$ crudely almost normal surfaces of weight at most $K(C+1)$.  Connect two vertices by an edge if they differ either by a pinch or a move of type 0), 1), 2a) or 2b) or their inverses.  Call a vertex of $\mG$ corresponding to the surface $T$ \emph{obviously weakly reducible} with respect to $\Delta_2$ if there exists (possibly equal) $\kappa_0,\kappa_1\in \Delta_2$  and essential curves $\alpha_0\subset T\cap st(\kappa_0)$ and $\alpha_1\subset T\cap st(\kappa_1)$ such that $\alpha_0$ and $\alpha_1$ respectively compress to opposite sides of $T$.  Note that by Haken's normal surface algorithm it is decidable if these subsurfaces of $T$ compress to one side or the other.  Indeed, one can pass to a further subdivision $\Delta_4$ such that if such a subsurface compresses to one side, then there exists an essential compressing disc normal with respect to $\Delta_4$ which is a fundamental solution to the appropriate normal surface equations.

By construction each $S_i$ is $16\delta_4$-locally incompressible and of $\Delta_3$ weight $\le K(C+1)$.  Since weakly reducible implies reducible \cite{CaGo} in non Haken 3-manifolds, it suffices to apply the following Lemma \ref{reducible weeding}, to weed out the reducible $S_i$'s.  Simply remove those $S_i$'s lying  in the same component of $\mG$ as an obviously weakly reducible splitting.  

Once reducible splittings are eliminated from the list, duplications are eliminated by applying Lemma \ref{duplicate weeding}.  Again $S_i$ and $S_j$ are isotopic if and only if they are in the same component of $\mG$.\qed

\begin{lemma} \label{reducible weeding}  i)  If $T$ is obviously weakly reducible with respect to $\Delta_2$, then it is weakly reducible. 

ii)  If T is transverse to $\Delta_2$ and $\delta_3$-bi-compressible, then $T$ is obviously weakly reducible with respect to $\Delta_2$.   

iii)  If $T$ is $16\delta_4$-locally incompressible and weight$_{\Delta_2}(T)\le L$, then $T$ is weakly reducible if and only if it lies in the same component of $\mG$ as a  vertex that is obviously weakly reducible with respect to $\Delta_2$.  Here $L$ is the maximal $\Delta_3$ weight of all the $S_i$'s.
\end{lemma}

\begin{proof}  i)  This is immediate if $\inte(st(\kappa_0))\cap\inte(st(\kappa_1))=\emptyset$.  Otherwise these stars have a simplex $\lambda$ in common and hence $st(\kappa_0)\cup st(\kappa_1) \subset st^2(\lambda)$.  By construction $st^2(\lambda)$ is a 3-ball, so   some component $\alpha_3$ of $\partial (st^2(\lambda))\cap T $ is essential in $T$ (else $N$ is  the 3-sphere) and compresses to one side or the other.  Thus $\alpha_3$ and one of $\alpha_0$ or $\alpha_1$ provide a weak reduction.

ii)  Since $T $ is $\delta_3$-bi-compressible there exist essential compressing curves $\alpha_0$, $\alpha_1\subset T$ that compress to opposite sides and have diameter at most $\delta_3$ and hence lie in stars of simplices of $\Delta_2$. 

iii) If $T$ is weakly reducible, then by  Lemma \ref{thick isotopy} there exists a $C$-isotopy from $T=T_0$ to a $16\delta_4$-bi-compressible  surface $T''_1$ such that  each interpolating surface is $2\delta_4$-locally incompressible.  

By applying Lemma \ref{coarea}, with $\Delta=\Delta_3$ and $\epsilon=\delta_4/2$  we can assume that there exists a generic  $K(C+1)-\Delta_3$-isotopy $G$ from $T_0$ to $T'_1$, where $T'_1$ is transverse to $\Delta_3$, is $17\delta_4$-bi-compressible and each interpolating surface is $\delta_4$-locally incompressible.  Finally apply Lemma \ref{crude isotopy} to find a generic pinched crudely normal $K(C+1)-\Delta_3$-isotopy  $H$ from $T_0$ to $T_1$ where $T_1$ is crudely normal of weight $\le K(C+1)$ and $19\delta_4$-bi-compressible.  Since $19\delta_4<\delta_3$ it follows $T_1$ is obviously weakly reducible with respect to $\Delta_2$ and that  $H$ determines a path within $\mG$ from $T_0$ to $T_1$.  \end{proof}

\begin{lemma} \label{duplicate weeding}  If $S_i$ and $S_j$ are strongly irreducible then they are isotopic if and only if they are in the same component of $\mG$.\end{lemma}

\begin{proof}  If they are in the same component of $\mG$, then it is immediate that they are isotopic.  Conversely by Lemma \ref{thick isotopy} there exists a $C$-isotopy $F$ from $S_i$ to $S_j$ such that each interpolating surface is $2\delta_4$-locally incompressible.  By Lemma \ref{coarea} there exists a generic  $K(C+1)-\Delta_3$-isotopy from $S_i$ to $S_j$ such that each interpolating surface is $\delta_4$-locally incompressible.  By Lemma \ref{crude isotopy} there exists a generic pinched crudely normal $K(C+1)-\Delta_3$-isotopy $H$ from $S_i$ to $S_j$.  This $H$ determines a path in $\mG$ from $S_i$ to $S_j$.\end{proof}

\section{Appendix: Proof of Theorem \ref{minmax}}

The proof of the Min-Max Theorem \ref{minmax} is similar to that of the one-parameter case handled in \cite{CD}.  The following two things must be established:
\begin{enumerate}
\item  Starting from any given minimizing sequence $\{\Sigma_t\}^i$, one can "pull-tight" to produce a new minimizing sequence $\{\Gamma_t\}^i$ so that any min-max sequence obtained from $\{\Gamma_t\}^i$ converges to a stationary varifold.
\item At least one min-max sequence obtained from $\{\Gamma_t\}^i$ is \emph{almost minimizing (a.m.) in sufficiently small annuli}.  (see Definition 3.2 in \cite{CD} for the relevant definition).
\end{enumerate}

Given (1) and (2) it follows from \cite{CD} that the min-max limit is regular.  The genus bound \eqref{genusbound} then follows from \cite{K} since only the almost minimizing property is used there.   Thus (1) and (2) together imply the Min-Max Theorem \ref{minmax}.  Item (1) follows from straightforward modifications of \cite{CD}.   We will give a proof of (2) since the combinatorial argument is more involved than the one-parameter case handled in \cite{CD}.
\\
\\
\emph{Proof of (2).}
The idea of the proof is that if \emph{no} min-max sequence were almost minimizing, we can glue together several isotopies to produce a competitor sweepout with all areas strictly below $W(\Pi,M)$ and thus violate the definition of $W(\Pi,M)$.  To accomplish this, we will need to find many disjoint annuli on which to pull down the slices, which requires a combinatorial argument due to Almgren-Pitts \cite{P}.

Given $x\in M$, for $r,s>0$, denote by $An(x,r,s)$ the open annulus centered around $x$ with outer radius $s$ and inner radius $r$.  Given sequences $\{r_i\}_{i=1}^k$ and $\{s_i\}_{i=1}^k$ with $s_i<r_{i+1}$, consider the family of annuli $\{An(x,r_i,s_i)\}_{i=1}^{k}$.  Such a family is called \emph{admissible} if $r_{i+1}>2s_i$ for each $i$.   A family of admissible annuli containing precisely $L$ annuli is called \emph{$L$-admissible}.  Given a surface $\Sigma\subset M$ and  an $L$-admissible family $\mathcal{F}$, we will say that $\Sigma$ is $\epsilon$-almost minimizing in an $\mathcal{F}$ is it $\epsilon$-almost minimizing in at least one of the annuli comprising $\mathcal{F}$.

\begin{lemma}\label{almostmin}
There exists an integer $L=L(n)$ and a min-max sequence $\Sigma_j$ converging to a stationary varifold such that  $\Sigma_j$ is almost minimizing in every $L$-admissible family of annuli. 
\end{lemma}
\begin{proof}
Much of the proof of Lemma \ref{almostmin} follows \cite{CD}, and thus we focus only on the parts that are different.  The proof is by 
contradiction.  If it failed, we would obtain a set $S\subset [0,1]^n$ of "large slices" and an open covering $\{\mathcal{O}_i\}_{i=1}^M$ of $S$ so 
that to each open set $\mathcal{O}_i$ is associated an $L$-admissible family $An_{\mathcal{O}_i}$.  It is enough to prove that there is another 
cover $\{\mathcal{U}_i\}$ of $S$ refining $\{\mathcal{O}_i\}_{i=1}^M$ (in the sense that each element in the collection $\{\mathcal{U}_i\}$ is 
contained in at least one of the $\mathcal{O}_i$) such that:
\begin{enumerate}
\item To each $\mathcal{U}_i$ we can fix one annulus $An_i$ which is one of the annuli comprising $An_{\mathcal{O}_j}$ for some $\mathcal{O}_j$ containing $\mathcal{U}_i$.
\item Each $\mathcal{U}_i$ intersects at most $d=d(n)$ other elements in the collection $\{\mathcal{U}_i\}$
\item To each $\mathcal{U}_i$ we can associate a compactly supported $C^\infty$ function $\phi_i:\mathcal{U}_i\rightarrow [0,1]$.  Moreover, for any $x\in S$ we have that $\phi_i(x)=1$ for some $i$.
\item If $\phi_i(x)$ and $\phi_j(x)$ are both nonzero for some $i$ and $j$, then $An_i\cap An_j=\emptyset$
\end{enumerate}

To achieve this, it will help to have some notation (following Pitts). For each $j\in\mathbb{N}$, denote by $I(1,j)$ be the cell complex on the interval $I^1=[0,1]$ whose 1-cells are $[0,3^{-j}], [3^{-j},2\times 3^{-j}],...[1-3^{-j},1]$ and whose 0-cells are $[0],[3^{-j}],...[1]$.  We are considering the $n$-dimensional cell complex $I^n$, which we can write as a tensor product: $$I(n,j)=I(1,j)\otimes I(1,j)\otimes...\otimes I(1,j)\mbox{       ($n$ times).} $$ \noindent For $0\leq p\leq n$ we define a $p$-cell as an element $\alpha=\alpha_1\otimes\alpha_2...\otimes\alpha_n$ where $\sum_{i=1}^n \dim(\alpha_i)=p$.  The sets $\mathcal{U}_i$ will be "thickened" $p$-cells for suitably large $j$.  Precisely, for each $p$ cell $\alpha\in I(n,j)$ expressed as $\alpha=\alpha_1\otimes\alpha_2...\otimes\alpha_n$, the \emph{thickening} of $\alpha$ denoted $T(\alpha)$ is the open set obtained by replacing each $\alpha_i$ in its tensor expansion that is a 0-cell $[x]$ by $[x-3^{-j-1},x+3^{-j-1}]$.  Note that for an $n$-cell $\beta$, $T(\beta)=\beta$.  Denote the collection of thickened cells by $T(I(n,j))$.   For some $j_0$ large enough we have that any element in $T(I(n,j_0))$ has diameter smaller than the Lebesgue number of the covering $\{\mathcal{O}_i\}$ and thus each element of  $T(I(n,j_0))$ is contained in some $\mathcal{O}_i$.  We set our refinement $\mathcal{U}_i$ to then be the collection $T(I(n,j_0))$.
Because each thickened cell $T(\alpha)$ only intersects those faces $\beta$ such that $\alpha$ and $\beta$ are faces of a common cell $\gamma$, we easily obtain (2).  For  (1), (3), and (4), we can now invoke the following combinatorial lemma of Pitts:

\begin{lemma}\label{pittscombinatorial}(Proposition 4.9 in \cite{P})
For each $\sigma\in I(n,j)$, let $A(\sigma)$ be an $({3^n})^{3^n}$-admissible family of annuli.  Then there exists a function $F$ which assigns to each $\sigma\in I(n,j)$ an element $F(\sigma)\in A(\sigma)$ such that $A(\sigma)$ and $A(\tau)$ are disjoint whenever $\sigma$ and $\tau$ are faces (possibly of different dimensions) of some common cell $\gamma\in I(n,j)$.
\end{lemma}

To apply Lemma \ref{pittscombinatorial} to our setting, set $L=({3^n})^{3^n}$ and for each cell $\alpha\in I(n,j_0)$ we choose some $L$-admissible family $A(\alpha)$ from one of the $\mathcal{O}_i$ containing it (it does not matter which).  This then determines a map $A$ and the lemma applies to give an annulus associated to each cell $\alpha$ and thus to each to thickened cell $T(\alpha)$ satisfying (4).  Item (3) then follows easily.
\end{proof}

While Lemma \ref{almostmin} gives an almost minimizing property in many annuli about each point, we need the property for all annuli sufficiently small: 

\begin{lemma}\label{allsmall}
Some subsequence of the min-max sequence produced by Lemma \ref{almostmin} is almost minimizing in sufficiently small annuli.  
\end{lemma}\noindent
\begin{proof}  We first show the following claim:\\ \\
(1*) Given a point $p\in M$, and a min-max sequence $\Sigma_j$ that is a.m. in all $L$-admissible familes of annuli about $p$, one can find a positive number $r(p)$ and a subsequence of $\Sigma_j$ that is a.m.~ in sufficiently small annuli of outer radius at most $r(p)$ about $p$.\\ \\
To prove (1*), first fix an $L$-admissible family $\mathcal{F}$ of annuli about $p$. By the pigeonhole priniciple $\Sigma_j$ is a.m.~ in one of the annuli comprising $\mathcal{F}$ for infinitely many $j$.  Fix the outermost such annulus $A(r,s)$ and a pass to the subsequence (not relabeled) of the $\Sigma_j$ that are a.m.~ in $A(r,s)$.  Now consider the family $\mathcal{F}'$ consisting of the annuli in $\mathcal{F}$ exterior to $A(r,s)$, the annulus $A(r/2,s)$ as well as the annuli in $\mathcal{F}$ interior to $A(r,s)$ whose inner and outer radii are multiplied by $1/2$ (such a family is still admissible).  By the a.m.~ property infinitely many of the $\Sigma_j$ must be a.m.~ in one of the annuli in the family $\mathcal{F}'$.  Again choose the outermost such annulus (which can be no further out than $A(r/2,s)$ by construction) and pass to a subsequence a.m. in this new annulus.  Iterating this procedure and a diagonal argument gives a min-max sequence a.m.~ in $\emph{all}$ annuli sufficiently small about $p$, proving (1*).

To prove Lemma \ref{allsmall}, we combine (1*) with a Vitali-type covering argument. For each $p\in M$ and given a min-max sequence $\Sigma_j$, let $m(p)$ denote the supremum of positive numbers $\eta$ so that some subsequence of $\Sigma_j$ is a.m.~ in annuli with outer radius $\eta$.  By (1*), $m(p)>0$.   Set $r_m(p):=m(p)/2$.  Choose some $p_1\in M$ with $r_m(p_1)>\frac{1}{2}\sup_{q\in M} r_m(q)$.  Then pass to a subsequence of $\Sigma_j$ (not relabeled) that is a.m.~ in annuli with outer radius at most $r_m(p_1)$.  Choose then $p_2\in M\setminus B_{r_m(p_1)}(p_1)$ so that 
\begin{align}
r_m(p_2)>\frac{1}{2}\sup_{q\in M\setminus B_{r_m(p_1)}(p_1)} r_m(q)
\end{align} 
and pass to a further subsequence of $\Sigma_j$ a.m.~ in annuli about $p_2$ and outer radius at most $r_m(p_2)$.  Iterating this procedure gives rise in finitely many steps to a cover of the entire manifold with the desired properties since by the maximality of the construction and (1*) it cannot happen that $r_m(p_i)\rightarrow 0$.  \end{proof}

\enddocument